\documentclass[12pt]{article}%
\usepackage[intlimits]{amsmath}
\usepackage{amssymb}
\usepackage{latexsym}
\usepackage{anysize}
\usepackage{subfigure}
\usepackage[T1]{fontenc}
\usepackage[sc]{mathpazo}
\usepackage{color}
\usepackage[colorlinks]{hyperref}
\usepackage{amsfonts}
\usepackage{microtype}
\usepackage{graphicx}%
\definecolor {refcol}{RGB}{40,0,255}
\hypersetup{colorlinks=true,allcolors=refcol}
\makeatletter
\newfont{\footsc}{cmcsc10 at 8truept}
\newfont{\footbf}{cmbx10 at 8truept}
\newfont{\footrm}{cmr10 at 10truept}
\newtheorem{theorem}{Theorem}

\newtheorem{corollary}[theorem]{Corollary}

\newtheorem{problem}[theorem]{Problem}
\newtheorem{proposition}[theorem]{Proposition}

\newenvironment{proof}[1][Proof]{\noindent{\textbf {#1}  }}  {\hfill$\Box$\bigskip}
\begin{document}

\title{\textbf{A note on the positive semidefinitness of }$A_{\alpha}(G)$}
\author{Vladimir Nikiforov\thanks{Department of Mathematical Sciences, University of
Memphis, Memphis TN 38152, USA.} \ and Oscar Rojo\thanks{Department of
Mathematics, Universidad Cat\'{o}lica del Norte, Antofagasta, Chile.}}
\date{}
\maketitle

\begin{abstract}
Let $G$ be a graph with adjacency matrix $A(G)$ and let $D(G)$ be the diagonal
matrix of the degrees of $G$. For every real $\alpha\in\left[  0,1\right]  $,
write $A_{\alpha}\left(  G\right)  $ for the matrix
\[
A_{\alpha}\left(  G\right)  =\alpha D\left(  G\right)  +(1-\alpha)A\left(
G\right)  .
\]
Let $\alpha_{0}\left(  G\right)  $ be the smallest $\alpha$ for which
$A_{\alpha}(G)$ is positive semidefinite. It is known that $\alpha_{0}\left(
G\right)  \leq1/2$. The main results of this paper are:

(1) if $G$ is $d$-regular then
\[
\alpha_{0}=\frac{-\lambda_{\min}(A(G))}{d-\lambda_{\min}(A(G))},
\]
where $\lambda_{\min}(A(G))$ is the smallest eigenvalue of $A(G)$;

(2) $G$ contains a bipartite component if and only if $\alpha_{0}\left(
G\right)  =1/2$;

(3) if $G$ is $r$-colorable, then $\alpha_{0}\left(  G\right)  \geq1/r$.

\end{abstract}



\section{Introduction}

Let $G$ be a graph with adjacency matrix $A(G)$, and let $D\left(  G\right)  $
be the diagonal matrix of its vertex degrees. In \cite{Nik16}, it was proposed
to study the family of matrices $A_{\alpha}(G)$ defined for any real
$\alpha\in\left[  0,1\right]  $ as
\[
A_{\alpha}(G)=\alpha D(G)+(1-\alpha)A(G).
\]
Since $A_{0}\left(  G\right)  =A\left(  G\right)  $ and $2A_{1/2}\left(
G\right)  =Q\left(  G\right)  $, where $Q\left(  G\right)  $ is the signless
Laplacian of $G$, the matrices $A_{a}$ can help to study subtle relations
between $A\left(  G\right)  $ and $Q\left(  G\right)  $.

A major distinction between $Q\left(  G\right)  $ and $A\left(  G\right)  $ is
the fact that $Q\left(  G\right)  $ is positive semidefinite, whereas
$A\left(  G\right)  $ is not, except if $G$ is empty. Thus, given $G$, it is
natural to ask for which $\alpha\in\left[  0,1\right]  $ is $A_{\alpha}(G)$
positive semidefinite. To further discuss this question we need the following
notation: given a square matrix $M$ with real eigenvalues, we write
$\lambda\left(  M\right)  $ and $\lambda_{\min}\left(  M\right)  $ for the
largest and the smallest eigenvalues of $M$.

In \cite{Nik16}, some general results on the matrices $A_{\alpha}(G)$ have
been proved. In particular, if $1\geq\alpha>\beta\geq0,$ observe that
\[
A_{\alpha}\left(  G\right)  -A_{\beta}\left(  G\right)  =\left(  \alpha
-\beta\right)  L\left(  G\right)  ,
\]
where $L\left(  G\right)  =D\left(  G\right)  -A\left(  G\right)  $ is the
Laplacian of $G$. Hence, using Weyl's inequalities (see Theorem W below), one
can get the following propositions:

\begin{proposition}
\label{mono} Let $1\geq\alpha>\beta\geq0$. If $G$ is a graph of order $n$ with
$A_{\alpha}(G)=A_{\alpha}$, $A_{\beta}(G)=A_{\beta}$, and $L\left(  G\right)
=L$, then
\begin{equation}
\left(  \alpha-\beta\right)  \lambda\left(  L\right)  >\lambda_{\min
}(A_{\alpha})-\lambda_{\min}(A_{\beta})\geq0. \label{eq1}%
\end{equation}
If $G$ is connected, then the right inequality in (\ref{eq1}) is strict.
\end{proposition}

\begin{proposition}
\label{pdef} If $\alpha\geq1/2$, then $A_{\alpha}(G)$ is positive
semidefinite. If $\alpha>1/2$ and $G$ has no isolated vertices, then
$A_{\alpha}(G)$ is positive definite.
\end{proposition}

In particular, inequalities (\ref{eq1}) imply that for every $G$, the function
$f_{G}\left(  \alpha\right)  :=\lambda_{\min}\left(  A_{\alpha}\left(
G\right)  \right)  $ is continuous and nondecreasing in $\alpha$. Thus, there
is a smallest value $\alpha\in\left[  0,1\right]  $ such that $\lambda_{\min
}\left(  A_{\alpha}\left(  G\right)  \right)  =0$. Denote this value by
$\alpha_{0}\left(  G\right)  $ and note that $A_{\alpha}(G)$ is positive
semidefinite if and only if $\alpha_{0}\left(  G\right)  \leq\alpha\leq1$.

Having $\alpha_{0}\left(  G\right)  $ in hand, we restate a problem that has
been raised in \cite{Nik16}:

\begin{problem}
\label{pro1} Given a graph $G$, find $\alpha_{0}\left(  G\right)  $.
\end{problem}

This problem seems difficult in general, but is worth studying, because
$\alpha_{0}\left(  G\right)  $ relates to various structural parameters of
$G.$ In this paper, we find $\alpha_{0}\left(  G\right)  $ if $G$ is regular
or $G$ contains a bipartite component. We also give a lower bound on
$\alpha_{0}\left(  G\right)  $ of $r$-colorable graphs.\medskip

Let us start by observing a general property of $\alpha_{0}\left(  G\right)  $
that helps to extend results from connected to general graphs:

\begin{proposition}
\label{disc}If $G$ is a disconnected graph, then
\[
\alpha_{0}\left(  G\right)  =\max\left\{  \alpha_{0}\left(  H\right)  :H\text{
is a component of }G\right\}  .
\]

\end{proposition}

Studying the matrices $A_{\alpha}(G)$ is particularly easy if $G$ is regular;
e.g., we can calculate $\alpha_{0}\left(  G\right)  $ as follows:

\begin{proposition}
If $G$ is a $d$-regular graph, then
\[
\alpha_{0}\left(  G\right)  =\frac{-\lambda_{\min}(A(G))}{d-\lambda_{\min
}(A(G))}.
\]

\end{proposition}

\begin{proof}
Obviously, if $G$ is $d$-regular of order $n$, then
\[
A_{\alpha}(G)=\alpha dI_{n}+(1-\alpha)A(G),
\]
where $I_{n}$ is the identity matrix of order $n$. Hence, we see that
\[
\lambda_{\min}(A_{\alpha}(G))=\alpha d+(1-\alpha)\lambda_{\min}(A(G)).
\]
The right side of this identity increases with $\alpha$, so $\alpha_{0}\left(
G\right)  $ is the unique solution of the equation
\[
\alpha d+(1-\alpha)\lambda_{\min}(A(G))=0,
\]
which gives the required identity.
\end{proof}

\section{Bipartite graphs}

The main result in this section is Corollary \ref{bipgen} below, which
characterizes all graphs with $\alpha_{0}\left(  G\right)  =1/2$. First, we
consider the case of connected graphs:

\begin{proposition}
\label{bip}A connected graph $G$ is bipartite, if and only if $\alpha
_{0}\left(  G\right)  =1/2$.
\end{proposition}

\begin{proof}
Since $G$ is connected, Proposition \ref{mono} implies that $\lambda_{\min
}\left(  A_{\alpha}\left(  G\right)  \right)  $ is increasing in $\alpha$;
hence $\lambda_{\min}\left(  A_{\alpha}\left(  G\right)  \right)  =0$ if and
only if $\alpha=\alpha_{0}\left(  G\right)  $. \ Next, recall that a connected
graph $G$ is bipartite if and only if $\lambda_{\min}\left(  Q\left(
G\right)  \right)  =0,$ that is, $\lambda_{\min}\left(  A_{1/2}\left(
G\right)  \right)  =0$. Therefore $G$ is bipartite if and only if $\alpha
_{0}\left(  G\right)  =1/2.$
\end{proof}

Using Proposition \ref{disc}, we extend Proposition \ref{bip} to all nonempty
graphs, i.e., graphs containing edges:

\begin{corollary}
\label{bipgen}If $G$ is a nonempty graph, then $\alpha_{0}\left(  G\right)
=1/2$ if and only if $G$ has a bipartite component.
\end{corollary}

Next, we give explicit upper bounds on $\alpha_{0}\left(  G\right)  $ if $G$
has no bipartite components. To this end, recall a simplified version of
Weyl's inequalities for the eigenvalues of Hermitian matrices (see, e.g.
\cite{HoJo85}, p. 181): \medskip

\textbf{Theorem W} \emph{Let }$A$\emph{\ and }$B$\emph{\ be Hermitian matrices
of order }$n,$\emph{\ and let their eigenvalues be indexed in descending
order.\ If }$1\leq k\leq n$,\emph{\ } \emph{then}%
\begin{equation}
\lambda_{k}\left(  A\right)  +\lambda_{n}\left(  B\right)  \leq\lambda
_{k}\left(  A+B\right)  \leq\lambda_{k}\left(  A\right)  +\lambda_{1}\left(
B\right)  . \label{Wes}%
\end{equation}

Returning to our goal, suppose that $n\geq3$ and let $E_{n}$ denote the graph
on $n$ vertices obtained by identifying a vertex of a triangle with an
endvertex of $P_{n-2}$, the path of order $n-2$.

\begin{proposition}
\label{talfa1} If $G$ is a connected nonbipartite graph of order $n$, then
\begin{equation}
\alpha_{0}\left(  G\right)  \leq\frac{-\lambda_{\min}(A(G))}{\lambda_{\min
}(Q(G))-2\lambda_{\min}(A(G))}<\frac{\sqrt{\lfloor{n/2}\rfloor\lceil
{n/2}\rceil}}{\frac{1}{12n^{2}}+2\sqrt{\lfloor{n/2}\rfloor\lceil{n/2}\rceil}}.
\label{alfa1}%
\end{equation}

\end{proposition}

\begin{proof}
Let $\alpha\leq1/2$. Applying Weyl's inequality (\ref{Wes}) to the identity
\[
A_{\alpha}(G)=\alpha Q(G)+(1-2\alpha)A(G),
\]
we obtain%
\[
\lambda_{\min}(A_{\alpha}(G))\geq\alpha\lambda_{\min}(Q(G))+(1-2\alpha
)\lambda_{\min}(A(G)).
\]
Hence, if
\[
1/2\geq\alpha\geq\frac{-\lambda_{\min}(A(G))}{\lambda_{\min}(Q(G))-2\lambda
_{\min}(A(G))},
\]
then
\[
\lambda_{\min}(A_{\alpha}(G))\geq\alpha\lambda_{\min}(Q(G))+(1-2\alpha
)\lambda_{\min}(A(G))\geq0,
\]
implying that
\begin{equation}
\alpha_{0}\left(  G\right)  \leq\frac{-\lambda_{\min}(A(G))}{\lambda_{\min
}(Q(G))-2\lambda_{\min}(A(G))}. \label{in1}%
\end{equation}

To prove the second inequality in (\ref{alfa1}), recall a result of Constatine
\cite{Con85}: \emph{if }$G$\emph{ is a graph of order }$n$\emph{, then
}$\lambda_{\min}\left(  A(G)\right)  >-\sqrt{\lfloor{n/2}\rfloor\lceil
{n/2}\rceil}$, \emph{unless }$G$\emph{ is the complete bipartite graph
}$K_{\lfloor{n/2}\rfloor\lceil{n/2}\rceil}$\emph{. }

Similarly, in \cite{CCRS08}, Cardoso et al. have proved that if $G$ is a
connected nonbipartite graph of order $n$, then $\lambda_{\min}(Q(G))>\lambda
_{\min}(Q(E_{n})),$ unless $G=E_{n}$.

Now the second inequality in (\ref{alfa1}) follows by combining these two
results with the bound $\lambda_{\min}(Q(E_{n}))\geq n^{-2}/12,$ which has
been proved in \cite{CRS07}.
\end{proof}

Using Proposition \ref{disc}, we extend Proposition \ref{talfa1} to arbitrary
graphs as follows:

\begin{proposition}
\label{talfa2} If $G$ is a graph of order $n$ and has no bipartite components,
then
\[
\alpha_{0}\left(  G\right)  \leq\frac{-\lambda_{\min}(A(G))}{\lambda_{\min
}(Q(G))-2\lambda_{\min}(A(G))}<\frac{\sqrt{\lfloor{n/2}\rfloor\lceil
{n/2}\rceil}}{\frac{1}{12n^{2}}+2\sqrt{\lfloor{n/2}\rfloor\lceil{n/2}\rceil}%
}.
\]

\end{proposition}

We omit the details of the proof, which is based on the facts that $n^{-2}/12$
is decreasing in $n$ and that
\[
\frac{x}{c+2x}%
\]
is increasing in $x$ if $c>0$ and $x>0$.

\section{The chromatic number and $\lambda_{\min}\left(  A_{\alpha}\left(
G\right)  \right)  $}

Connections between $\alpha_{0}\left(  G\right)  $ and the chromatic number of
$G$ exist also for nonbipartite graphs. We state one such relation in
Corollary \ref{cron} below, which is deduced from a more general result,
stated in Theorem \ref{tH}.

Recall that in \cite{LOAN11}, it has been shown that if $G$ is a graph of
order $n,$ size $m$, and chromatic number $\chi,$ then
\[
\lambda_{\min}\left(  Q\left(  G\right)  \right)  \leq\frac{\left(
\chi-2\right)  2m}{\left(  \chi-1\right)  n}.
\]
It turns out that this bound can be extended to all matrices $A_{\alpha
}\left(  G\right)  $ as follows:

\begin{theorem}
\label{tH}Let $G$ be a graph of order $n,$ size $m$, and chromatic number
$\chi.$ If $\alpha\in\left[  0,1\right]  $, then
\begin{equation}
\lambda_{\min}\left(  A_{\alpha}\left(  G\right)  \right)  \leq\frac{\left(
\alpha\chi-1\right)  2m}{\left(  \chi-1\right)  n}. \label{ubo}%
\end{equation}

\end{theorem}

\begin{proof}
Suppose that $\alpha\in\left[  0,1\right]  $, and set $A_{\alpha}:=A_{\alpha
}\left(  G\right)  $ and $\lambda_{\min}:=\lambda_{n}\left(  A_{\alpha
}\right)  $. Let $V:=V\left(  G\right)  $ and suppose that $V_{1}%
,\ldots,V_{\chi}$ are the color classes of $G$. For every $k\in\left\{
1,\ldots,\chi\right\}  $, set
\[
e_{k}=\sum_{u\in V_{k}}d\left(  u\right)  .
\]

Our main goal is to prove that for every $k\in\left\{  1,\ldots,\chi\right\}
$,%
\begin{equation}
\lambda_{\min}\left(  A_{\alpha}\right)  \left(  \chi\left(  \chi-2\right)
\left\vert V_{k}\right\vert +n\right)  \leq2m+\left(  \alpha\chi-2\right)
\chi e_{k}. \label{main}%
\end{equation}
To this end, take an integer $k\in\left\{  1,\ldots,\chi\right\}  $ and define
a vector $\mathbf{x}:=\left(  x_{1},\ldots,x_{n}\right)  $ by%
\[
x_{i}:=\left\{
\begin{array}
[c]{ll}%
\chi-1, & \text{if }i\in V_{k}\text{;}\\
-1, & \text{otherwise.}%
\end{array}
\right.
\]
Calculating the quadratic form $\left\langle A_{\alpha}\mathbf{x}%
,\mathbf{x}\right\rangle $ for the vector $\mathbf{x}$, we get%
\begin{equation}
\left\langle A_{\alpha}\mathbf{x},\mathbf{x}\right\rangle =\alpha\sum_{u\in
V}x_{u}^{2}d\left(  u\right)  +2\left(  1-\alpha\right)  \sum_{\left\{
u,v\right\}  \in E\left(  G\right)  }x_{u}x_{v}. \label{qf}%
\end{equation}
To find the first term in the right side of (\ref{qf}), note that
\begin{align*}
\sum_{u\in V}x_{u}^{2}d\left(  u\right)   &  =\sum_{u\in V\backslash V_{k}%
}d\left(  u\right)  +\sum_{u\in V_{k}}\left(  \chi-1\right)  ^{2}d\left(
u\right) \\
&  =2m+\chi\left(  \chi-2\right)  e_{k}.
\end{align*}
Similarly, to find the second term in the right side of (\ref{qf}), note that%
\begin{align*}
\sum_{\left\{  u,v\right\}  \in E\left(  G\right)  }x_{u}x_{v}  &
=\sum_{\left\{  u,v\right\}  \in E\left(  G\right)  ,u\in V_{k},v\notin V_{k}%
}x_{u}x_{v}+\sum_{\left\{  u,v\right\}  \in E\left(  G\right)  ,u\notin
V_{k},v\notin V_{k}}x_{u}x_{v}\\
&  =\left(  1-\chi\right)  e_{k}+\left(  m-e_{k}\right)  =m-\chi e_{k}.
\end{align*}
Combining these two identities, we obtain
\begin{align*}
\left\langle A_{\alpha}\mathbf{x},\mathbf{x}\right\rangle  &  =\alpha\left(
2m+\chi\left(  \chi-2\right)  e_{k}\right)  +2\left(  1-\alpha\right)  \left(
m-\chi e_{k}\right) \\
&  =2m+\left(  \alpha\chi-2\right)  \chi e_{k}.
\end{align*}
On the other hand, Rayleigh's principle implies that
\[
\lambda_{\min}\left(  A_{\alpha}\right)  \left\Vert \mathbf{x}\right\Vert
^{2}\leq\left\langle A_{\alpha}\mathbf{x},\mathbf{x}\right\rangle .
\]
Hence, we complete the proof of (\ref{main}) by noting that%
\[
\left\Vert \mathbf{x}\right\Vert ^{2}=\left(  \chi-1\right)  ^{2}\left\vert
V_{k}\right\vert +\left(  n-\left\vert V_{k}\right\vert \right)  =\chi\left(
\chi-2\right)  \left\vert V_{k}\right\vert +n.
\]

Finally, adding inequalities (\ref{main}) for all $k\in\left\{  1,\ldots
,\chi\right\}  $, we get%
\[
\lambda_{\min}\sum_{k=1}^{\chi}\left(  \chi\left(  \chi-2\right)  \left\vert
V_{k}\right\vert +n\right)  \leq\sum_{k=1}^{\chi}\left(  2m+\left(  \alpha
\chi-2\right)  \chi e_{k}\right)  ,
\]
implying that%
\[
\lambda_{\min}\left(  \chi\left(  \chi-2\right)  n+\chi n\right)  \leq2\chi
m+2\left(  \alpha\chi-2\right)  \chi m.
\]
Now inequality (\ref{ubo}) follows by simple algebra.
\end{proof}

\begin{corollary}
\label{cron}If $G$ is an $r$-colorable graph, then $\alpha_{0}\left(
G\right)  \geq1/r$.\footnote{Note that Corollary \ref{cron} has been proved
for regular graphs in \cite{Nik16}.}
\end{corollary}

Both Theorem \ref{tH} and Corollary \ref{cron} are tight. Indeed, let $G$ be a
complete regular $r$-partite graph. Writing $n$ and $m$ for the order and the
size of $G$, it is not hard to see that
\[
\lambda_{\min}\left(  A_{\alpha}\left(  G\right)  \right)  =\alpha
\frac{\left(  r-1\right)  n}{r}-\left(  1-\alpha\right)  \frac{n}{r}%
=\frac{\left(  \alpha r-1\right)  2m}{\left(  r-1\right)  n},
\]
which implies that $\alpha_{0}\left(  G\right)  =1/r.$

\end{document}